\documentclass[]{svmult}
\usepackage{float} 
\usepackage{graphics}
\usepackage{caption}
\usepackage{subcaption}
\usepackage{import} 
\usepackage{amssymb}
\usepackage[abs]{overpic}
\usepackage[cmyk]{xcolor}
\usepackage{amsmath, amssymb}
\title*{Independence complexes of circle graphs}

\author{Rhea Palak Bakshi \and Ali Guo \and Dionne Ibarra \and Gabriel Montoya-Vega \and Sujoy Mukherjee \and Marithania Silvero \and Jonathan Spreer} 
\authorrunning{Bakshi, Guo, Ibarra, Montoya-V., Mukherjee, Silvero, Spreer} 

\institute{Rhea Palak Bakshi\at Department of Mathematics, University of California, Santa Barbara, USA \\
\email{rheapalak@math.ucsb.edu $|$ rheapalakbakshi@gmail.com} \and 
Ali Guo \at {The George Washington University} \\
\email{hguo30@gwmail.gwu.edu} \and
Dionne Ibarra \at School of Mathematics, Monash University, Australia \\
\email{dionne.ibarra@monash.edu} \and
Gabriel Montoya-Vega \at Department of Mathematics, The Graduate Center CUNY, USA and Department of Mathematics, University of Puerto Rico at R\'io Piedras, PR, USA \\
\email{gabrielmontoyavega@gmail.com} \and Sujoy Mukherjee \at Department of Mathematics, University of Denver, USA \\
\email{sujoymukherjee.math@gmail.com $|$ sujoy.mukherjee@du.edu} \and
Marithania Silvero \at Departamento de Álgebra, Universidad de Sevilla, Spain \\
\email{marithania@us.es} \and 
Jonathan Spreer \at School of Mathematics and Statistics, University of Sydney, Australia\\
\email{jonathan.spreer@sydney.edu.au}}

\begin{document}

\maketitle
\abstract{Independence complexes of circle graphs are purely combinatorial objects. However, when constructed from some diagram of a link $L$, they reveal topological properties of $L$, more specifically, of its Khovanov homology. We analyze the homotopy type of independence complexes of circle graphs, with a focus on those arising when the graph is bipartite. Moreover, we compute (real) extreme Khovanov homology of a $4$-strand pretzel knot using chord diagrams and independence complexes.}

\section{Introduction}
A \emph{chord diagram} is a circle with a finite set of chords with disjoint boundary points. A graph is a \textit{circle graph}, if it is the intersection graph for some chord diagram. Given a simple graph $G$, its independence complex $I(G)$ contains a simplex for each subset of pairwise non-adjacent vertices. Due to work by Gonz\'alez-Meneses, Manch\'on, and one of the authors in \cite{GMS}, it is possible to extract topological information of a link by studying the independence complex of certain graphs built from its diagrams. In particular, given a diagram $D$ of a link $L$, we can apply a particular smoothing to its crossings to obtain a disjoint union of circles with chords denoting where the crossings were smoothed. These circles can be viewed as chord diagrams if the chords connecting different circles are ignored. The cohomology of the independence complex of the associated circle graph is isomorphic to the extreme Khovanov homology of the link diagram. Combinatorially studying homotopy types of independence complexes translates into studying extreme Khovanov homology, which is the fact motivating this article.

Figure~\ref{fig:pipeline} illustrates a summary of the construction process, starting with a diagram of the French coat of arms link and leading to an independence complex whose homotopy type is $S^1 \vee S^1$.

\begin{figure}[ht]
    \centering
    \includegraphics[height=1.3cm]{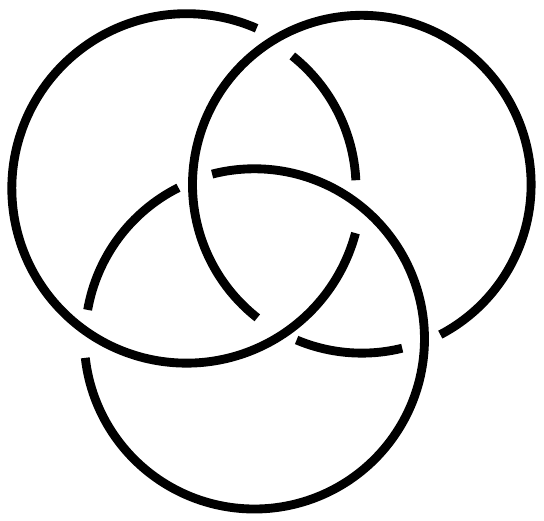} 
    \raisebox{.6cm}{\, $\to$ \,}
\includegraphics[height=1.3cm]{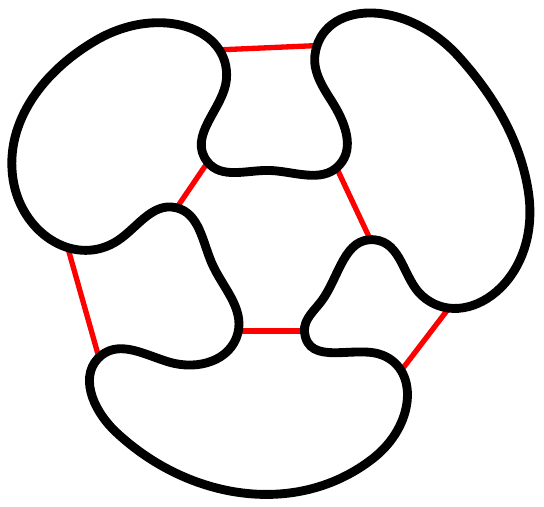} 
    \raisebox{.6cm}{\, $\to$ \,}
\includegraphics[height=1.3cm]{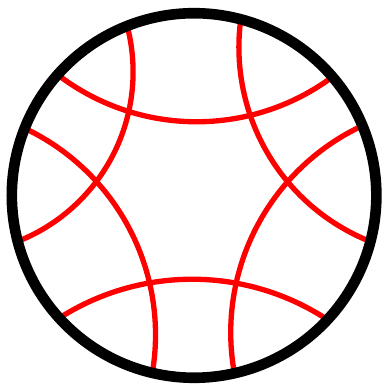}
    \raisebox{.6cm}{\, $\to$ \,}
    \raisebox{.15cm}{\includegraphics[height=1.0cm]{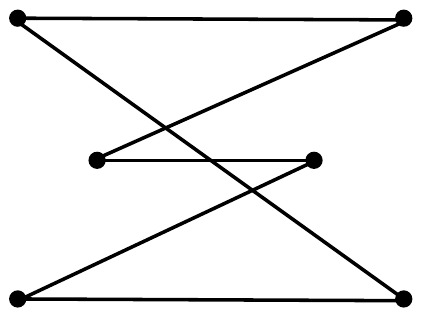} }
    \raisebox{.6cm}{\, $\to$ \,}
    \raisebox{.15cm}{\includegraphics[height=1.0cm]{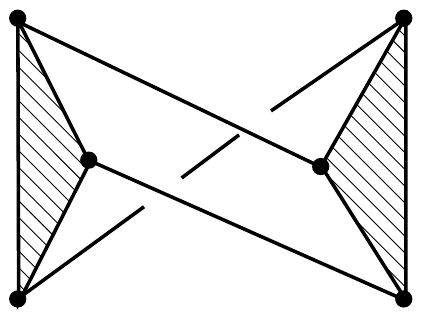}}
    \raisebox{.6cm}{\ $\to$ $S^1 \vee S^1$}

    \caption{From left to right: Link diagram $D$, its $A$-state, chord diagram $\mathcal{C}_D$, intersection graph $G(D)$, independence complex $I(D)$, and homotopy type.}
    \label{fig:pipeline}
\end{figure}

Our starting point is the following conjecture: 

\begin{conjecture}[Przytycki-Silvero \cite{PrzytyckiSilvero18}]
  \label{conj:general}
  If $G$ is a circle graph, then $I(G)$ has the homotopy type of a wedge of spheres.
\end{conjecture}

Note that in Conjecture~\ref{conj:general}, a wedge of spheres is thought as an arbitrary number of spheres iteratively built up by attaching a sphere to an existing wedge of spheres along an arbitrary point. In particular, any disjoint union of spheres is a wedge of spheres in this definition, as can be seen when starting this process with a number of $0$-dimensional spheres.

However, if the attaching point is required to be the same for all spheres, as is the case in the definition of wedge sum in \cite[Chapter 0]{HatcherAlgTop}, then Conjecture~\ref{conj:general} should be rephrased by substituting {\it{``wedge of spheres''}} by {\it{``disjoint union of wedges of spheres''}}, as shown in Section~\ref{sec:properties}.

There is extensive literature devoted to the study of the independence complexes arising from families of graphs, and how their homotopy types are modified when performing specific transformations in these graphs; see, for example, \cite{Barmak, Csorba09, Jonsson, Kozlov, Nagel-Reiner, PrzytyckiSilvero18, PrzytyckiSilvero24}. Although, these references support Conjecture~\ref{conj:general}, the techniques developed in them do not seem to be strong enough to prove it in full generality.

As already mentioned, Conjecture~\ref{conj:general} was originally motivated by the study of Khovanov homology, a bigraded homological invariant of knots and links categorifying the Jones polynomial \cite{Kho1}. It can be shown that the graphs arising from a link diagram are bipartite, and therefore, in the context of Khovanov homology, the interest of Conjecture \ref{conj:general} is restricted to the family of bipartite circle graphs. Hence, one can consider the following weaker version of Conjecture \ref{conj:general}.

\begin{conjecture}[Przytycki-Silvero \cite{PrzytyckiSilvero18}]
  \label{conj:bipartite}
  If $G$ is a bipartite circle graph, then $I(G)$ has the homotopy type of a wedge of spheres.
\end{conjecture}

This article is organised as follows. In Section~\ref{sec:setup}, we review the relation of extreme Khovanov homology with the theory of independence simplicial complexes. In Section~\ref{sec:properties}, we go through a number of topological properties of independence complexes of bipartite circle graphs. As an application,  Section~\ref{sec:examples} contains detailed computations of extreme Khovanov homology for a diagram of a four-strand pretzel knot.

\subsection*{Acknowledgements}

The content of this article summarizes the work authors did during the workshop {\em Low Dimensional Topology: Invariants of Links, Homology Theories, and Complexity} held at the Matrix Research Institute (The University of Melbourne) in June 2024. Bakshi was supported by Dr. Max Rössler, the Walter Haefner Foundation, and the ETH Zürich Foundation. Ibarra was supported by the Australian Research Council (ARC) under the discovery project scheme, grant numbers DP210103136 and DP240102350. Montoya-V. was supported by the Matrix-Simons travel grant and acknowledges the support of the National Science Foundation through Grant DMS-2212736. Mukherjee was supported by the Matrix-Simons travel grant and is grateful to P. Vojt\v{e}chovsk\'{y} for his support through the Simons Foundation Mathematics and Physical Sciences Collaboration Grant for Mathematicians no. 855097. Silvero was partially supported by the Spanish Research Grant PID2020-117971GB-C21 funded by MCIN/AEI/10.13039/501100011033. Spreer was supported by the ARC under the discovery project scheme, grant number DP220102588.

\section{Khovanov homology and independence complexes}
\label{sec:setup}

Khovanov homology is a homological link invariant, introduced in \cite{Kho1} as a categorification of the Jones polynomial. This invariant has shown to be quite powerful. For instance, it detects the unknot, both the right and left handed trefoils, the figure-eight knot, and the knot $5_1$ \cite{BaHuSi, BaDoLeLiSa, BaSi, KrMr}.

Let $D$ be a link diagram. For every $i,j \in \mathbb{Z}$, Khovanov introduced a bigraded chain complex $$\ldots \, \longrightarrow \, C^{i,j}(D) \, \stackrel{\partial_i}{\longrightarrow} \, C^{i+1,j}(D) \, \longrightarrow \ldots,$$
where the chain groups $C^{i,j}(D)$ are non-trivial for a finite number of $j\in \mathbb{Z}$. The homology groups associated to that complex, $Kh^{i,j}(D)$ are link invariants, and we refer to them as the \textit{Khovanov homology} groups of $D$ (or, more generally, of the link $L$ represented by $D$). 
We define

\[j_{\min}(D) = \min \{j \, | \, C^{*,j}(D) \mbox{ is non-trivial}\} \]

\noindent as the smallest $j$-grading for which the chain group is non-trivial.
Note that the value of $j_{\min}(D)$ (and the Khovanov complex itself) is not a link invariant, and, in fact, can differ for two diagrams representing the same link. However, we know that $Kh^{*,j}(L)$ is trivial when $j \leq j_{\min}(D)$, for any diagram $D$ representing $L$. 

In other words, the value $j_{\min}(D)$ is the lowest index $j$ for which Khovanov homology {\em can} be non-trivial. We call the complex $\{C^{*,j_{\min}(D)}, \partial_i\}$ the \textit{extreme Khovanov complex} and its associated homology groups $Kh^{*, j_{\min}}(D)$ the \textit{extreme Khovanov homology groups} of $D$.

In \cite{GMS}, a geometric realization of extreme Khovanov homology was introduced. More precisely, it was shown that the extreme Khovanov homology groups of a link diagram coincide with the cohomology groups of a certain simplicial complex constructed from the diagram. We briefly review this construction now. 

Given a link diagram $D$, start by smoothing each crossing. For each crossing this can be done in two distinct ways. We choose for all crossings what is classically known as an \emph{$A$-smoothing}, which is illustrated in Figure~\ref{fig:smoothing}. This choice yields what is known as the \emph{$A$-state} of $D$. We write $\mathcal{C}_D$ for the resulting set of circles and chords. If we forget about those chords connecting different circles, then we can think of $\mathcal{C}_D$ as a finite family of disjoint chord diagrams $\mathcal{C}_1, \ldots, \mathcal{C}_k$. See Figure \ref{FigLink} for an example. 

\begin{figure}
    \centering
    \begin{subfigure}{.32\textwidth}
        \centering
         \raisebox{.2cm}{\includegraphics[height=1cm]{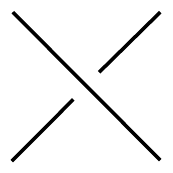}}
    \raisebox{.6cm}{\, $\to$ \,}
    \raisebox{.2cm}{\includegraphics[height=1cm]{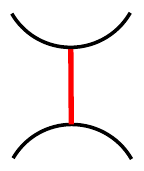}}
     \caption{$A$-smoothing ($A$-label).}
    \label{fig:smoothing}
    \end{subfigure}
     \begin{subfigure}{.32\textwidth}
        \centering
         $  \vcenter{\hbox{\begin{overpic}[scale=.15]{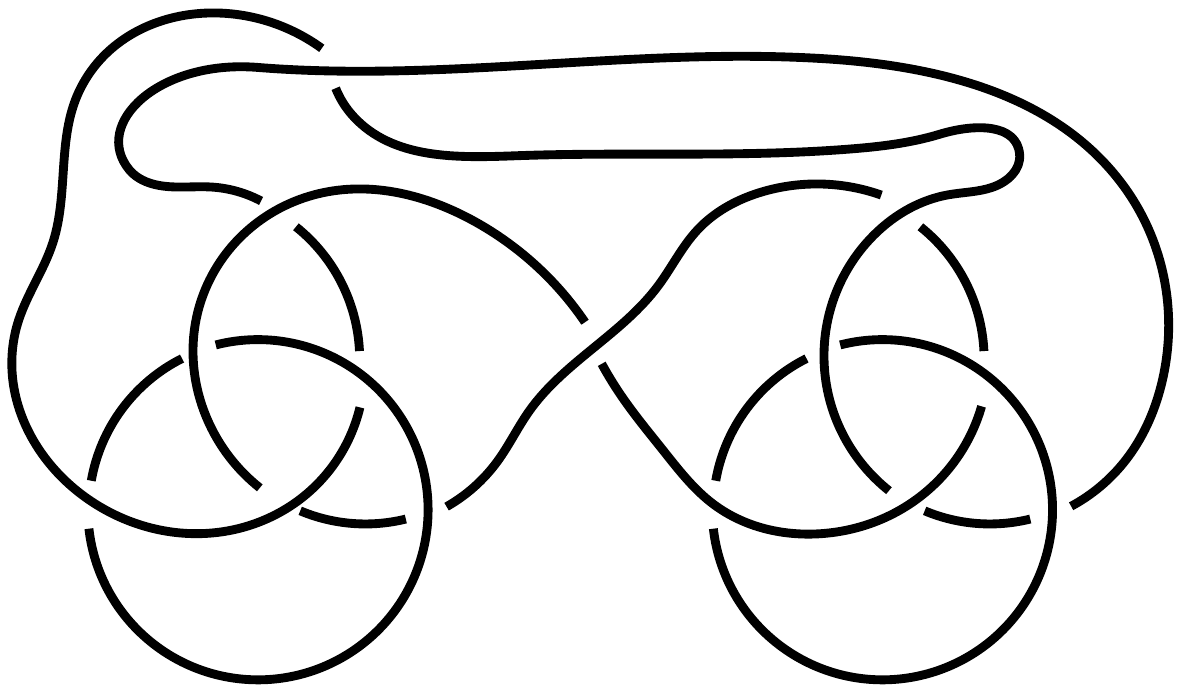}
\end{overpic}}}$
\caption{A link diagram $D$.}
    \label{fig:alinkexample}
    \end{subfigure}
     \begin{subfigure}{.32\textwidth}
        \centering
          $  \vcenter{\hbox{\begin{overpic}[scale=.15]{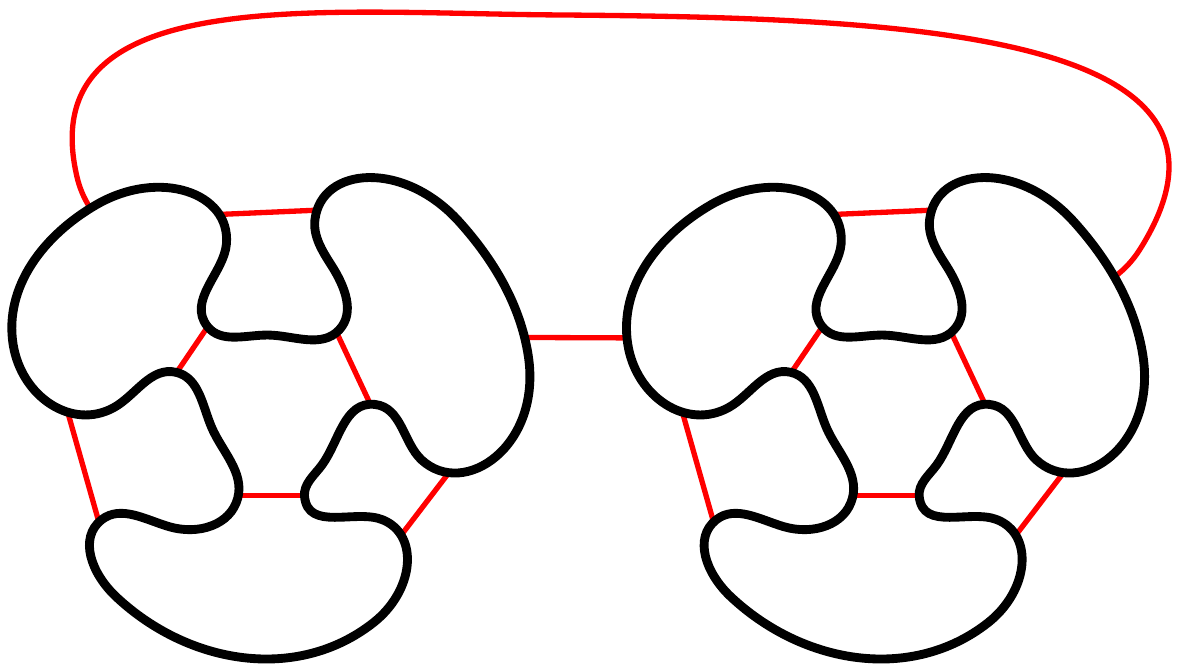}
\end{overpic}}}$
     \caption{Chord diagram $\mathcal{C}_D$.}
    \label{fig:achordexample}
    \end{subfigure}
    \caption{A link diagram and the chord diagram obtained by applying an $A$-smoothing to all of its crossings.}
    \label{FigLink}
\end{figure}

We write $G(D)$ and $I(D)$ for the intersection graph associated to $D$, and its independence complex, obtained respectively as $G(D) = G_{1} \sqcup \ldots \sqcup G_{k}$, and $I(D) = I(G_1) \ast \, \cdots \, \ast I(G_k)$, where $G_i$ denotes the graph $G_{\mathcal{C}_i}$, for $1 \leq i \leq k$, and $\ast$ is the join operation. When constructing the intersection graph $G_i$, its vertices are in bijection with the chords having both endpoints in $\mathcal{C}_i$ (drawn inside the circle), and edges between vertices if the corresponding chords intersect (or, equivalently, edges connecting vertices associated to two chords if their endpoints alternate along the circle in $\mathcal{C}_i$). By construction, the circle graph $G(D)$ is bipartite. 

The following result relates the extreme Khovanov homology of a link diagram and the independence complex arising from it. 

\begin{theorem}\label{TeoGMS}\cite{GMS}
Let $D$ be a link diagram with $n$ negative crossings. Then, the simplicial cochain complex of $I(D)$ is isomorphic to the extreme Khovanov complex 
shifted by $n-1$. As a consequence, $$Kh^{i,j_{\min}}(D) \simeq H^{i-1+n}(I(D)).$$ 
\end{theorem}
\section{Homotopy types of independence complexes of circle graphs}
\label{sec:properties}

In this section we present properties of independence complexes arising from circle graphs, with a focus on the specific case when we require the graph to be, additionally, bipartite.

Consider two chord diagrams $\mathcal{C}$ and $\mathcal{C}'$ with respective base points $c$ and $c'$ on their circles disjoint from the chords. Cutting the circles at $c$ and $c'$ in the projection plane, and regluing them in a crossing-free way yields the {\em connected sum} $\mathcal{C} \#_{c,c'} \mathcal{C}'$ of $\mathcal{C}$ and $\mathcal{C}'$ at $c$ and $c'$. Observe that, given two chord diagrams, the independence complex of the circle graph arising from their connected sum coincides with that of the graph associated to their disjoint union, i.e., \begin{equation}\label{join} I(G_{\mathcal{C} \#_{c,c'} \mathcal{C}'}) = I(G_{\mathcal{C} \sqcup \mathcal{C}'}) = I(G_\mathcal{C} \sqcup G_{\mathcal{C}'}) = I(G_{\mathcal{C}}) \ast I(G_{\mathcal{C}'}).\end{equation} In particular, it does not depend on the choice of basepoints $c$ and $c'$.  

Equation~\ref{join} highlights that we can construct new independence complexes coming from chord diagrams that are joins of previously known independence complexes, using the connected sum operation on the level of (chord) diagrams. This, in turn, implies that we can use the following well-known homotopy equivalences involving the join and wedge operations of spheres of different dimensions to produce independence complexes from chord diagrams with interesting homotopy types. For $n,m,k \in \mathbb{N}$, then 

\begin{equation}
 S^k \ast (S^n \vee S^m) = (S^k \ast S^n) \vee (S^k \ast S^m) = S^{k+n+1} \vee S^{k+m+1}; \label{eq:join1}
\end{equation}
\begin{equation}
    *_n (S^1 \vee S^1) = \bigvee_{2^n} S^{2n-1}. \label{eq:join2}
\end{equation}

For instance, starting with chord diagram $\mathcal{C}_a$ from Figure~\ref{fig:s1wedges1}, we can consider its connected sum to get $\mathcal{C}_c = \mathcal{C}_a \# \mathcal{C}_a$, as shown in Figure~\ref{fig:4wedges3}. Since $I(G_{\mathcal{C}_a}) \sim_h S^1 \vee S^1$, equations~\eqref{join} and ~\eqref{eq:join2} when $n=2$ allow us to directly compute the homotopy type of $I(G_{\mathcal{C}_c})$ as $\bigvee_4 S^3$.

\begin{figure}[ht]
        \centering
       \begin{subfigure}{.3\textwidth}
       \centering
     $$  \vcenter{\hbox{\begin{overpic}[scale=.3]{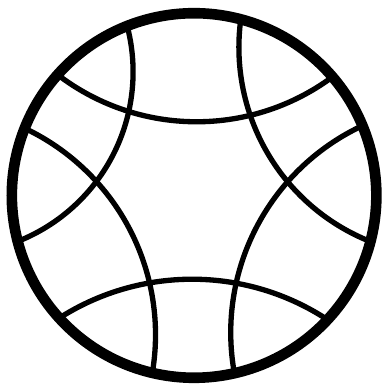}
\end{overpic}}}$$
\caption{$S^1 \vee S^1$}
        \label{fig:s1wedges1}
       \end{subfigure}
       \begin{subfigure}{.3\textwidth}
       \centering
       $$  \vcenter{\hbox{\begin{overpic}[scale=.3]{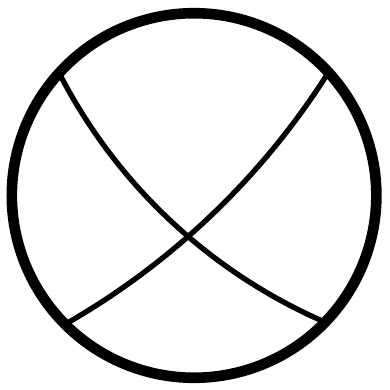}
\end{overpic}}}$$
            \caption{$S^0$}
        \label{fig:s1}
       \end{subfigure}
       \begin{subfigure}{.3\textwidth}
       \centering
       $$  \vcenter{\hbox{\begin{overpic}[scale=.2]{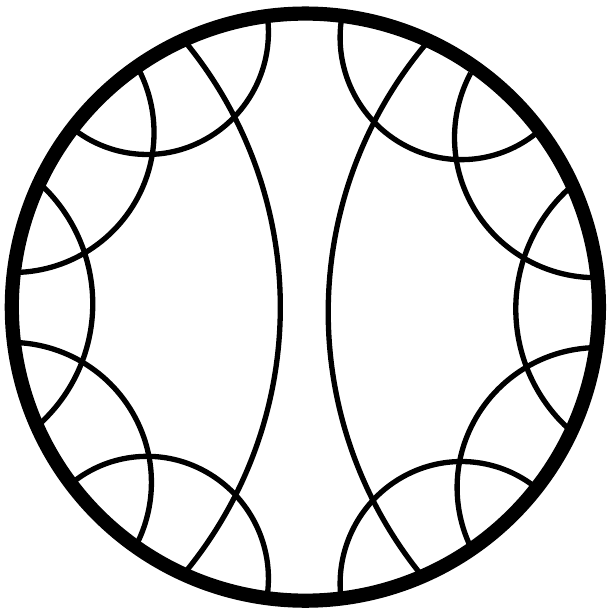}
\end{overpic}}}$$
            \caption{$ \bigvee_4 S^3$}
        \label{fig:4wedges3}
       \end{subfigure}
        \caption{Chord diagrams leading to bipartite circle graphs whose independence complexes are homotopic to the specified wedges of spheres.}
        \label{fig4}
    \end{figure}

We say that a chord diagram $\mathcal{C}$ is {\it{linear}} if it is possible to find two arcs $a$ and $b$ over the circle such that every chord in $\mathcal{C}$ has an endpoint in $a$ and the other endpoint in $b$. See Figure~\ref{fig:linearchord} for some examples. The circle graphs arising from linear chord diagrams are called {\em permutation graphs} in \cite{PrzytyckiSilvero18}.

\begin{figure}[ht]
        \centering
       \begin{subfigure}{.29\textwidth}
       \centering
     $  \vcenter{\hbox{\begin{overpic}[scale=.3]{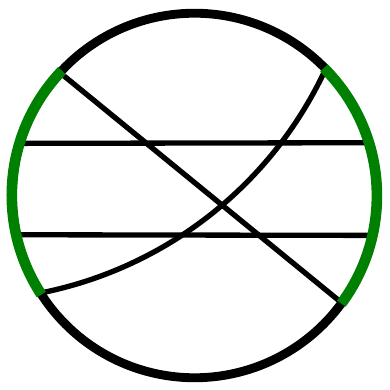}
\end{overpic}}}$
\caption{Linear chord diagram.}
        \label{fig:linear}
       \end{subfigure}
       \quad
       \begin{subfigure}{.29\textwidth}
       \centering
     $  \vcenter{\hbox{\begin{overpic}[scale=.3]{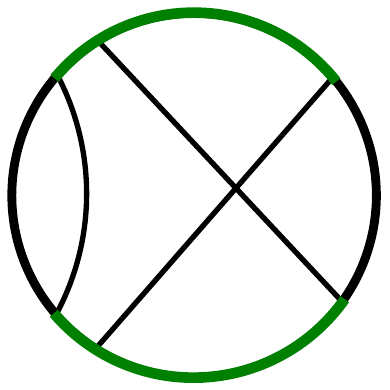}
\end{overpic}}}$
\caption{Linear chord diagram.}
        \label{fig:linear1}
       \end{subfigure}
       \quad
       \begin{subfigure}{.33\textwidth}
       \centering
       $  \vcenter{\hbox{\begin{overpic}[scale=.3]{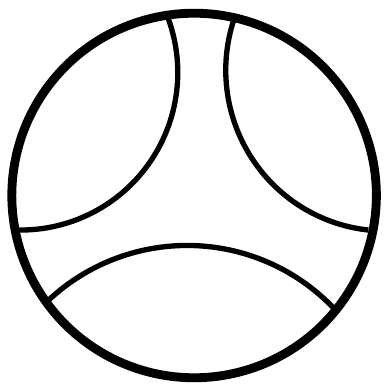}
\end{overpic}}}$
            \caption{Non-linear chord diagram.}
        \label{fig:nonlinear}
       \end{subfigure}
        \caption{Two linear chord diagrams, with the corresponding arcs $a$ and $b$ in green. The third chord diagram is not linear.}
    \label{fig:linearchord}
    \end{figure}

\begin{proposition}
\label{prop:disjointunion}
Let $\mathcal{C}$ and $\mathcal{C}'$ be two linear chord diagrams with associated independence complexes $I(\mathcal{C})$ and $I(\mathcal{C'})$. Then we can superimpose them to obtain a new linear chord diagram $\mathcal{D}$ such that the associated independence complex $I(\mathcal{D})$ has the homotopy type of $I(\mathcal{C}) \, \sqcup \, I(\mathcal{C'})$.
\end{proposition}

\begin{proof} \smartqed
First note that $\mathcal{C}$ and $\mathcal{C}'$ can be superimposed such that every chord of $\mathcal{C}$ intersects every chord of $\mathcal{C}'$. This can be achieved by interlacing the defining arcs of $\mathcal{C}$ and $\mathcal{C}'$ around a circle to obtain the chord diagram $\mathcal{D}$. By construction, $\mathcal{D}$ is linear. 

The resulting circle graph is $G(\mathcal{C}) \ast G(\mathcal{C'})$, and therefore the result follows from the definition of independence complex.
    \qed
\end{proof}

Moreover, we have the following result from \cite{PrzytyckiSilvero18}.

\begin{proposition}[Proposition 4.3 in \cite{PrzytyckiSilvero18}]
    \label{prop:wedge}
    For any wedge of spheres there exists a linear chord diagram whose independence complex is homotopy equivalent to it.
\end{proposition}

The proof of Proposition~\ref{prop:wedge} is constructive and very explicit. Combining Proposition~\ref{prop:disjointunion} and Proposition~\ref{prop:wedge} now yields the following proposition.

\begin{proposition}
For every $i= 1, \ldots, k$, let $X_i$ be an arbitrary wedge of spheres of (possibly) different dimensions, and consider the disjoint union $Y = X_1 \sqcup \ldots \sqcup X_k$. Then, there exists a circle graph $G_Y$ such that $I(G_Y)$ is homotopy equivalent to $Y$. 
\end{proposition}

\begin{proof}\smartqed
First, we apply Proposition~\ref{prop:wedge} to obtain a linear chord diagram $\mathcal{C}_i$ whose associated circle graph $G_i$ gives rise to $I(G_i) \sim_h X_i$, for every $1~\leq~i~\leq~k$. 
The result then holds by iteratively applying Proposition~\ref{prop:disjointunion}.
    \qed
\end{proof}

Note that if the definition of a wedge of an arbitrary number of spheres enforces a fixed point of identification, then Proposition~\ref{prop:disjointunion} yields an infinite number of counterexamples to Conjecture~\ref{conj:general}.

Observe that interlacing linear chord diagrams does not preserve the bipartite character of the graphs. As a result, some of the above observations cannot occur for chord diagrams coming from link diagrams -- which are bipartite. For this reason, in what follows we restrict our analysis to that of independence complexes arising from bipartite circle graphs.

Jonsson establishes in \cite[Theorem 1]{Jonsson} that independence complexes of bipartite circle graphs have the homotopy type of a suspension. Since wedges of spheres are suspensions, this is compatible with Conjecture~\ref{conj:bipartite}. The following result is a corollary of Jonsson's result, but we include a proof for completeness.

\begin{proposition}
  \label{disconnected}
The independence complex of a bipartite circle graph either has the homotopy type of $S^0$, or it is connected.
\end{proposition}

\begin{proof}\smartqed
The independence complex associated to a bipartite graph $G$ contains two cliques with vertex sets a partition of the entire set of vertices. If no other edge exists (that is, if every chord associated to one set of vertices intersect every chord associated to the other set of vertices) then its homotopy type is that of two points, i.e., $I(G) \sim_h S^0$. Any deviation from this situation forces additional edges in the independence complex which, hence, must connect the two cliques and therefore $I(G)$ is connected.  
    \qed
\end{proof}

In particular, Proposition~\ref{disconnected} implies the following statements.

\begin{corollary}
    There is no bipartite graph with independence complex homotopic to a wedge of a sphere $S^i$ and $S^0$, with $i \geq 0$.
\end{corollary}

\begin{corollary}
    The independence complex of a bipartite circle graph does not contain\footnote{By abuse of language, we identify the simplicial complex with its geometric realization.} three pairwise disjoint vertices.
\end{corollary}

Moving away from isolated vertices, we also have the following observation.

\begin{theorem}
    \label{prop:dim1}
    There is no bipartite circle graph with independence complex homotopic to a wedge of circles with more than two $S^1$ components.
\end{theorem}

\begin{proof} \smartqed
    We start with a couple of useful definitions. Let $G$ be a bipartite circle graph, with partition of the vertex set $V(G) = A \sqcup B$. The independence complex $I(G)$ of $G$  has two cliques spanned by the independent sets $A$ and $B$. We denote the clique spanned by $A$ and $B$ the {\em left} and {\em right} {\em cliques} $\Delta_R$ and $\Delta_L$, respectively. 
  Moreover, we call the $1$-skeleton of $I(G) \setminus \{ \Delta_R, \Delta_L \}$ (the independence complex minus the two large cliques) the {\em complement graph} $C(G)$.

  \medskip

  We claim that every $S^1$ component in $I(G)$ can be represented by a cycle of $4$-edges $\langle l,a,r,b \rangle$, where $l$ is an edge in $\Delta_L$, $r$ is an edge in $\Delta_r$, and $a,b$ are edges in distinct connected components of the complement graph $C(G)$.

  To see this, note that every sequence of edges in $\Delta_L$ or $\Delta_R$ can be shortened to a single edge, since $\Delta_L$ and $\Delta_R$ are cliques. Moreover, any sequence of two edges in the complement graph must span a triangle in the independence complex $I(G)$ and hence retract to a single edge in either $\Delta_L$ or $\Delta_R$. 
  
  Finally, $a$ and $b$ must be in separate connected components of $C(G)$. For, suppose they are not, then there is a sequence of edges in $C(G)$ connecting endpoints of $a$ and $b$. But this means that any consecutive pair of edges in this path span a triangle in $I(G)$ and hence $\langle l,a,r,b \rangle$ cannot generate homology.

  \medskip

  It now follows that any bipartite circle graph with independence complex homotopic to a wedge of circles with at least three $S^1$ components must have at least four connected components in its complement graph. We prove now that this is not possible.
  
  \medskip

  Assume that $G$ is a bipartite circle graph with complement graph $C(G)$, and assume that $C(G)$ has at least four connected components (each containing at least one edge). Pick four edges $a,b,c,d$ from four distinct connected components of $C(G)$. Call their endpoints $v_a,v_b,v_c,v_d \in \Delta_L$ and $w_a,w_b,w_c,w_d \in \Delta_R$ respectively.

  Recall that every vertex in $C(G)$ corresponds to a chord in the chord diagram realising $G$. The chord associated to $v_a$ must intersect the chords associated to $w_b$, $w_c$ and $w_d$. One of them, say, the one associated to $w_c$ must separate the ones associated to $w_b$ and $w_d$. Now, the chord associated to $v_c$ must intersect the chords associated to $w_b$ and $w_d$, but must be disjoint from the chord associated to $w_c$ (see Figure \ref{fig:cc} for an illustration), this is a contradiction to the chord associated to $w_c$ separating the chords associated to $w_b$ and $w_d$. Therefore, such a graph $G$ does not exist, and the statement holds.
  \qed 
\end{proof}
 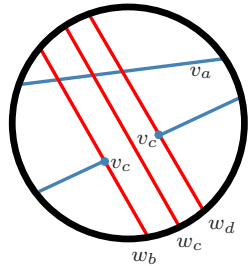
\begin{figure}[ht]
      \centering
    $ \vcenter{\hbox{\begin{overpic}[scale=.5]{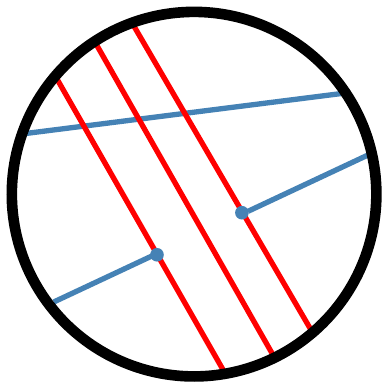}
       \put(70, 65){$v_a$}
        \put(40, 30){$v_c$}
        \put(50, 39){$v_c$}
        \put(48, -5){$w_b$}
         \put(65, 0){$w_c$}
          \put(77, 7){$w_d$}
\end{overpic}}}$
      \caption{A chord diagram illustrating last part of the proof of Proposition \ref{prop:dim1}.}
      \label{fig:cc}
  \end{figure}

Note that Theorem \ref{prop:dim1} is {\em best possible}: chord diagram in Figure~\ref{fig:s1wedges1} gives rise to an hexagon, which is bipartite, and whose independence complex has the homotopy type of $S^1 \vee S^1$.

\begin{remark}
    The complement graph $C(G)$ from the proof of Theorem~\ref{prop:dim1} can be used to give an alternative proof of \cite[Theorem 1]{Jonsson}: Slicing through $C(G)$ away from the endpoints of its edges yields a (possibly empty) polyhedral complex $P$. Since every vertex of the independence complex $I(G)$ is a vertex of $C(G)$, and lies in one of two cliques to the left and right of $P$, it follows that the homotopy type of the independence complex is $\Sigma P$.
\end{remark}

\section{An example: the case of pretzel knot $P(3,4,5,-5)$}
\label{sec:examples}

Inspired by results of \cite{3pretzels} we analyze (real) extreme Khovanov homology of pretzel knot $K=P(3,4,5,-5)$. The standard diagram $D$ of $K$, shown in Figure \ref{fig:P(3,4,5,-5)}, is $A$-adequate, meaning that the associated chord diagram $\mathcal{C}_D$ contains no chord with both endpoints in the same circle, and therefore the associated circle graph $G(D)$ is empty. By convention, we set $I(D)$ to be homotopy equivalent to $S^{-1}$, that we call sphere of dimension -1, and declare to satisfy $\Sigma S^{-1} = S^0$. 

By definition of the gradings in Khovanov homology, it holds that for any diagram $D$ we can compute $j_{min}(D) = p - 2n - k$, where $p$ and $n$ correspond to the number of positive and negative crossings of $D$, and $k$ is the number of circles in the associated chord diagram obtained when smoothing every crossing with an $A$-label. Hence, we get $j_{min}(D) = 12 - 2 \cdot 5 - 7 = -5$, and therefore 
$$Kh^{i, -5}(P(3,4,5,-5)) \simeq H^{i-1+5}(S^{-1}) = \left\{ \begin{array}{ll}
         \mathbb{Z} & \mbox{if $i = -5$};\\
 0 & \mbox{otherwise}.\end{array} \right.$$ 

Diagram $D$ was quite particular, since it was $A$-adequate. We devote the rest of this section to show a more illustrative example on how to compute extreme Khovanov homology by using the theory of independence complexes. 

\begin{figure}[ht]
    \centering
    \begin{subfigure}{.48\textwidth}
      \centering
    \includegraphics[scale=0.4]{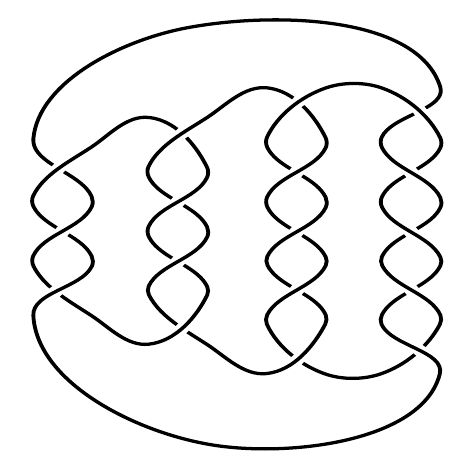}
    \subcaption{Diagram $D$ of $P(3,4,5,-5)$.}
    \label{fig:P(3,4,5,-5)}
    \end{subfigure} \quad
\begin{subfigure}{.48\textwidth}
\centering
    \includegraphics[scale=0.4]{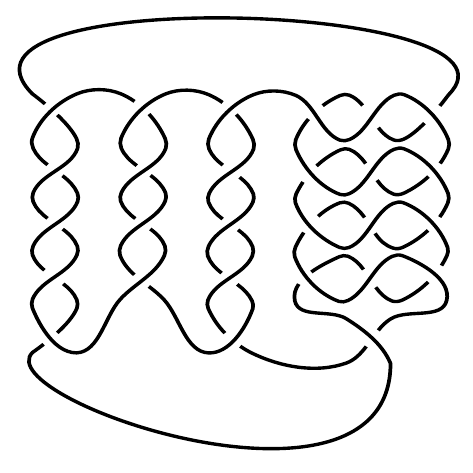}
 \subcaption{Diagram $D'$ of $P(3,4,5,-5)$.}
    \label{fig:P(3,4,5,-5)R2}
    \end{subfigure}
    \caption{Two diagrams of $P(3,4,5,-5)$ related by Reidemeister II moves.}
\end{figure}

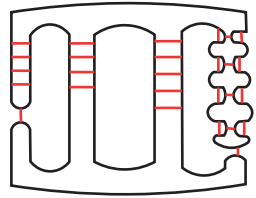
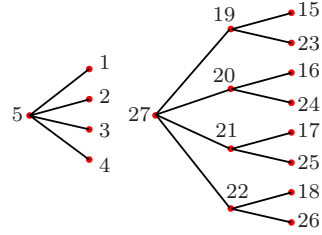
\begin{figure}[ht]
    \centering
    \begin{subfigure}{.49\textwidth}
    \centering
    \includegraphics[scale=0.1]{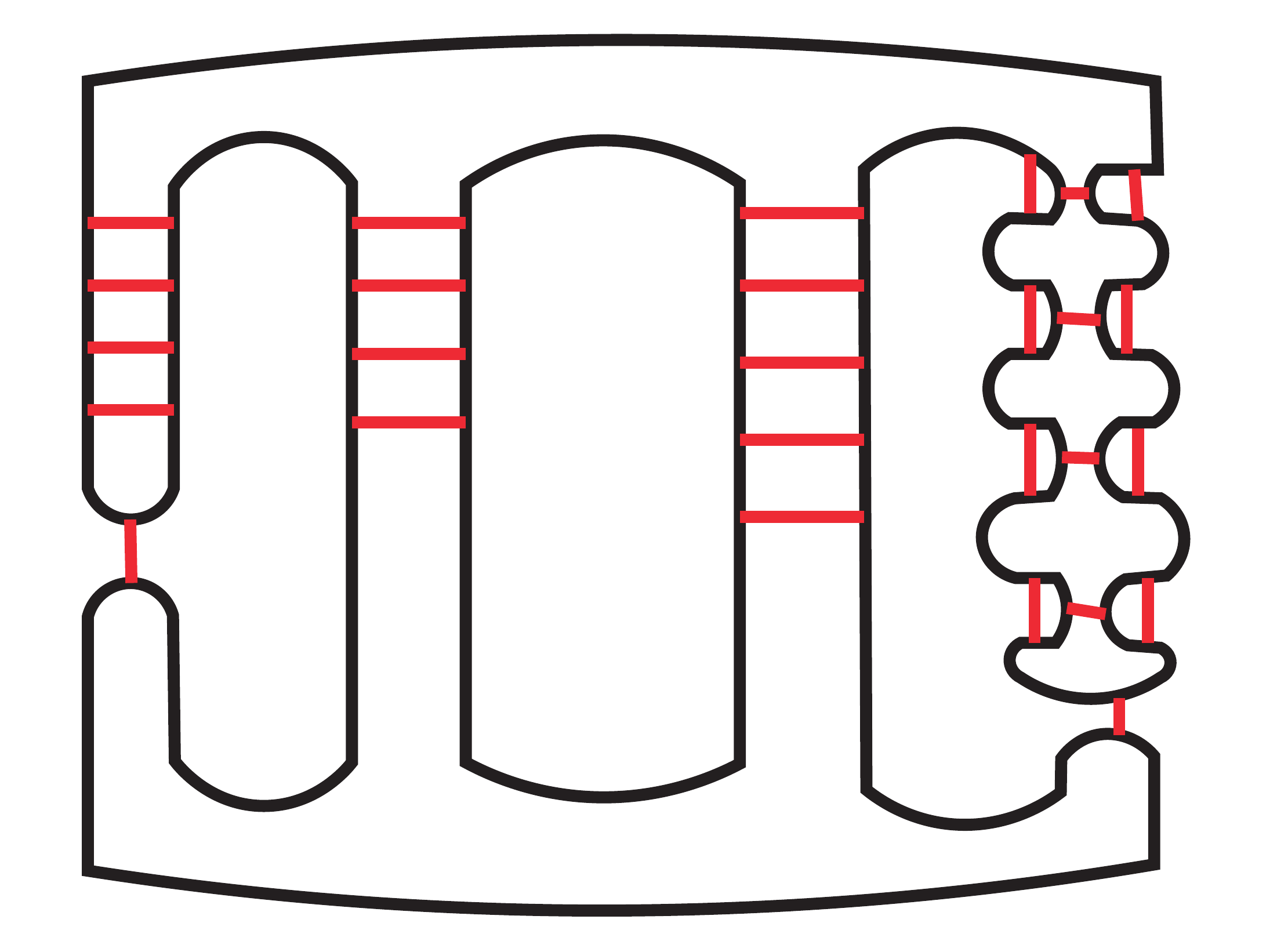}
    \subcaption{Chord diagram $\mathcal{C}_{D'}$.}
    \label{fig:P(3,4,5,-5)R2Asmoothing}
     \end{subfigure}
      \centering
    \begin{subfigure}{.49\textwidth}
    \centering
     $  \vcenter{\hbox{\begin{overpic}[scale=.2]{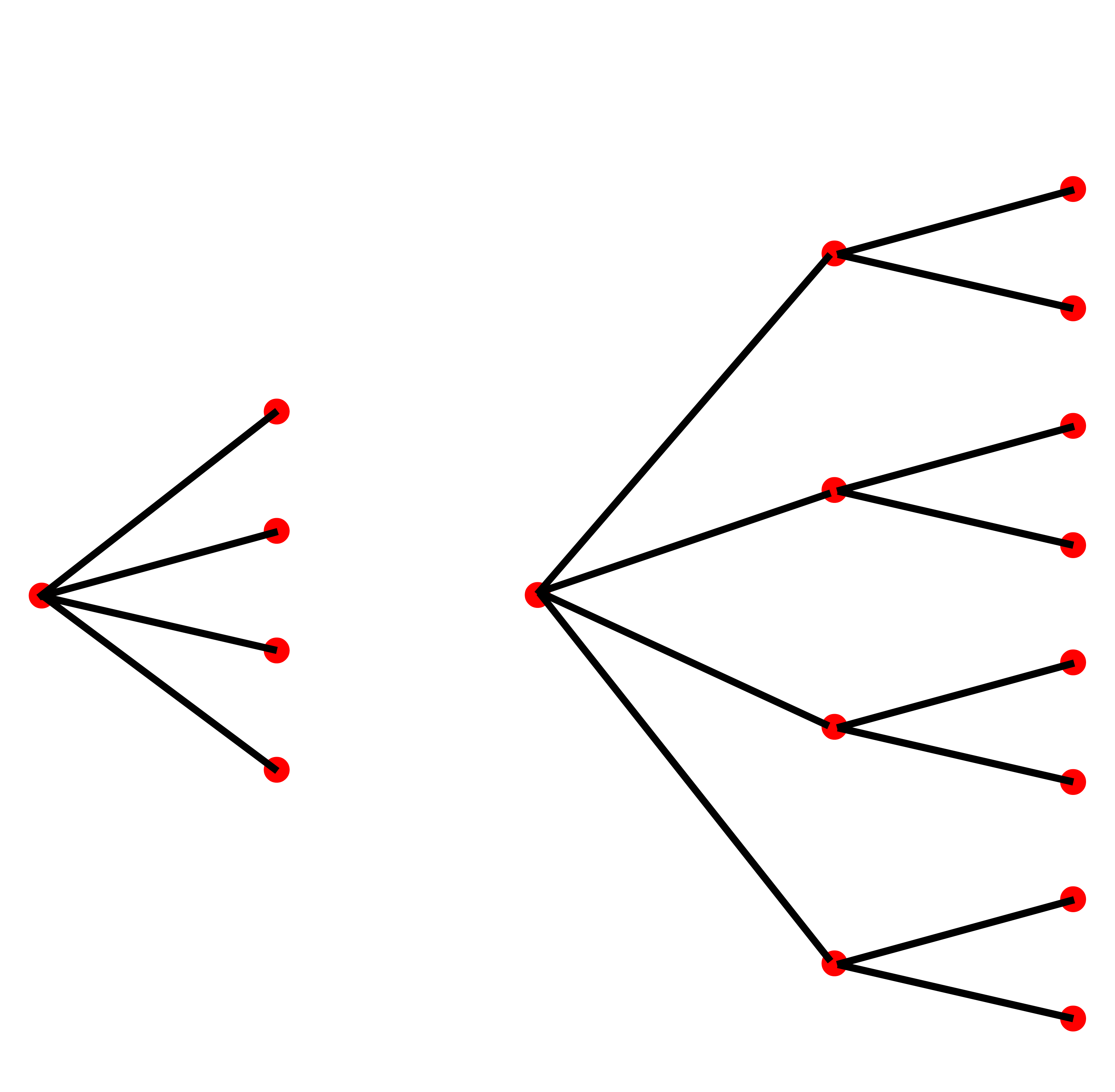}
      \put(30, 26){$4$}
      \put(30, 39){$3$}
      \put(30, 52){$2$}
      \put(30, 65){$1$}
      \put(-3, 45){$5$}
 \put(105, 5){$26$}
      \put(105, 16){$18$}
      \put(105, 27){$25$}
      \put(105, 39){$17$}
      \put(105, 49){$24$}
      \put(105, 62){$16$}
        \put(105, 72){$23$}
      \put(105, 85){$15$}
      \put(73, 83){$19$}
       \put(73, 60){$20$}
        \put(74, 39){$21$}
         \put(78, 16){$22$}
      \put(41, 45){$27$}
\end{overpic}}}$
    \subcaption{Circle graph $G(D')$.}
    \label{fig:P(3,4,5,-5)Landograph}
     \end{subfigure}
     \caption{The chord diagram and the Lando graph obtained from smoothing every crossing of $D'$ following an $A$-label.}
\end{figure}

Consider the diagram $D'$ of pretzel knot $P(3,4,5-5)$ shown in Figure \ref{fig:P(3,4,5,-5)R2}, obtained by performing some Reidemeister-$2$ moves to the standard diagram $D$. When we smooth every crossing following an $A$-label (as in Figure \ref{fig:smoothing}), we get a chord diagram containing a single circle and $26$ chords, as depicted in Figure \ref{fig:P(3,4,5,-5)R2Asmoothing}. The associated chord diagram $G(D)$ consists of the disjoint union of a star graph $G_1$ of $4$ rays and a binary tree $G_2$ with $8$ leaves, as shown in Figure \ref{fig:P(3,4,5,-5)Landograph}. Observe that vertices $1$ and $27$ dominate vertices $2$ and $15$ in the sense of Definition 3.1 in \cite{PrzytyckiSilvero18}. Therefore, Equation (\ref{join}) together with Domination Lemma \cite[Lemma 3.2]{PrzytyckiSilvero18} imply that
$$I(G) \sim_h I(G_1) \ast I(G_2) \sim_h \Sigma I(G_{\emptyset}) \ast \Sigma I(L_2 \sqcup L_2 \sqcup L_2) \sim_h \Sigma S^{-1} \ast \Sigma S^2 \sim_h S^4.$$

We can now compute $j_{min}(D') = 17 - 2 \cdot 10 - 2 = -5$, and as a consequence of Theorem \ref{TeoGMS} we get 
$$Kh^{i, -5}(P(3,4,5,-5)) \simeq H^{i-1+10}(S^4) = \left\{ \begin{array}{ll}
         \mathbb{Z} & \mbox{if $i = -5$};\\
 0 & \mbox{otherwise}.\end{array} \right.$$ 

Since the value of $j_{\min}$ coincides for both diagrams, computations for extreme Khovanov homology agree, as expected.

\bibliographystyle{spmpsci.bst}  
\bibliography{bib.bib} 

\begin{thebibliography}{10}
\providecommand{\url}[1]{{#1}}
\providecommand{\urlprefix}{URL }
\expandafter\ifx\csname urlstyle\endcsname\relax
  \providecommand{\doi}[1]{DOI~\discretionary{}{}{}#1}\else
  \providecommand{\doi}{DOI~\discretionary{}{}{}\begingroup
  \urlstyle{rm}\Url}\fi

\bibitem{BaHuSi}
Baldwin, J., Hu, Y., Sivek, S.: Khovanov homology and the cinquefoil.
\newblock J. Eur. Math. Soc.  (2024)

\bibitem{BaDoLeLiSa}
Baldwin, J.A., Dowlin, N., Levine, A.S., Lidman, T., Sazdanovic, R.: Khovanov
  homology detects the figure-eight knot.
\newblock Bull. Lond. Math. Soc. \textbf{53}(3), 871--876 (2021)

\bibitem{BaSi}
Baldwin, J.A., Sivek, S.: Khovanov homology detects the trefoils.
\newblock Duke Math. J. \textbf{171}(4), 885--956 (2022)

\bibitem{Barmak}
Barmak, J.A.: Star clusters in independence complexes of graphs.
\newblock Adv. Math. \textbf{241}, 33--57 (2013)

\bibitem{Csorba09}
Csorba, P.: Subdivision yields alexander duality on independence com- plexes.
\newblock Electr. J. Combin. \textbf{16}(2), Research paper 11 (2009)

\bibitem{GMS}
Gonz\'{a}lez-Meneses, J., Manch\'{o}n, P.M.G., Silvero, M.: A geometric
  description of the extreme {K}hovanov cohomology.
\newblock Proc. Roy. Soc. Edinburgh Sect. A \textbf{148}, 541--557 (2018)

\bibitem{HatcherAlgTop}
Hatcher, A.: Algebraic Topology.
\newblock Cambridge University Press, Cambridge (2002)

\bibitem{Jonsson}
Jonsson, J.: On the topology of independence complexes of triangle-free graphs.
\newblock Preprint (2011).
\newblock \urlprefix\url{https://people.kth.se/
  jakobj/doc/preprints/indbip.pdf}

\bibitem{Kho1}
Khovanov, M.: A categorification of the {J}ones polynomial.
\newblock Duke Math. J. \textbf{101}(3), 359–426 (2000)

\bibitem{Kozlov}
Kozlov, D.N.: Complexes of directed trees.
\newblock J. Combin. Theory Ser. A \textbf{88}(1), 112--122 (1999)

\bibitem{KrMr}
Kronheimer, P.B., Mrowka, T.S.: Khovanov homology is an unknot-detector.
\newblock Publ. Math. Inst. Hautes Études Sci. (113), 97--208 (2011)

\bibitem{Nagel-Reiner}
Nagel, U., Reiner, V.: Betti numbers of monomial ideals and shifted skew
  shapes.
\newblock Electron. J. Combin. \textbf{16}(2), Research Paper 3, 59 (2009)

\bibitem{3pretzels}
Oh, J., Siggers, M.H., Yang, S.Y., Yun, H.: On geometric realizations of
  extreme {K}hovanov homology of pretzel links.
\newblock arXiv:2401.06487 (2024)

\bibitem{PrzytyckiSilvero18}
Przytycki, J.H., Silvero, M.: Homotopy type of circle graph complexes motivated
  by extreme {K}hovanov homology.
\newblock Journal of Algebraic Combinatorics \textbf{48}, 119--156 (2018)

\bibitem{PrzytyckiSilvero24}
Przytycki, J.H., Silvero, M.: Khovanov homology, wedges of spheres and
  complexity.
\newblock RACSAM \textbf{118}(102) (2024)

\end{thebibliography}

\end{document}